\documentclass[11pt]{amsart}
\usepackage[margin=1in]{geometry}
\usepackage[utf8]{inputenc}
\usepackage[T1]{fontenc}
\usepackage{amsmath,amssymb,amsthm,bm,graphicx,enumitem}
\usepackage[colorlinks=true,linkcolor=blue,citecolor=blue,urlcolor=blue,breaklinks=true]{hyperref}

\newtheorem{theorem}{Theorem}[section]

\newtheorem{proposition}{Proposition}[section]
\newtheorem{corollary}{Corollary}[section]
\theoremstyle{remark}
\newtheorem{remark}{Remark}[section]
\theoremstyle{definition}
\newtheorem{definition}{Definition}[section]
\newtheorem{example}{Example}[section]
\numberwithin{equation}{section}

\newcommand{\arcsinh}{\operatorname{arcsinh}}

\title{New Explicit Examples of Mannheim Pairs in $\mathbb{R}^3$}

\author{Jos\'e Alan}
\address{Universidade Federal Rural de Pernambuco (UFRPE), Dept. de Matem\'atica, Rua Manoel de Medeiros, s/n, Dois Irm\~aos, Recife, PE, 52171-900, Brazil.}
\email{josealan.santos@ufrpe.br}

\author{Renato Teixeira}
\address{Universidade Federal Rural de Pernambuco (UFRPE), Dept. de Matem\'atica, Rua Manoel de Medeiros, s/n, Dois Irm\~aos, Recife, PE, 52171-900, Brazil.}
\email{renato.teixeira@ufrpe.br}

\author{Thiago Mendon\c{c}a}
\address{Universidade Federal Rural de Pernambuco (UFRPE), Dept. de Matem\'atica, Rua Manoel de Medeiros, s/n, Dois Irm\~aos, Recife, PE, 52171-900, Brazil.}
\email{thiago.mendonca@ufrpe.br}


\begin{document}

\begin{abstract}
In this paper, we construct new explicit parametrizations of Mannheim pairs in Euclidean 3-space. Unlike the known examples in the literature, which are essentially slant helices or general helices, the curve presented here is neither a helix nor a slant helix. We provide the full Frenet apparatus, the explicit Cartesian coordinates expressed in terms of elementary functions, and the corresponding Mannheim partner curve. The construction is based on a specific choice of the angular function that governs the curve's tangent indicatrix, leading to integrable expressions.
\end{abstract}

\maketitle
\section{Introduction}
The study of curves with special geometric properties is a central topic in classical differential geometry. Among these, Mannheim curves---introduced by A. Mannheim in 1878 \cite{Mannheim}---are defined by the existence of a partner curve such that, at corresponding points, the principal normal line of the first curve coincides with the binormal line of the second. If $\alpha$ and $\alpha^*$ are such curves, the pair $(\alpha, \alpha^*)$ is called a Mannheim pair.

It is well known \cite{Liu, OrbayKasap} that a curve $\alpha$, with curvature $\kappa>0$ and torsion $\tau$, is a Mannheim curve if and only if there exists a nonzero constant $\lambda$ such that $\kappa = \lambda \big( \kappa^2 + \tau^2\big)$. 
The literature contains several characterizations and examples of Mannheim curves. Liu and Wang \cite{Liu} established necessary and sufficient conditions for a curve to be a Mannheim partner in Euclidean and Minkowski spaces. Orbay and Kasap \cite{OrbayKasap} derived relationships between the curvatures and torsions of the two curves in a Mannheim pair. More recently, \c{C}ay and Yaylı \cite{Alper} studied slant helices that are also Mannheim curves, showing that such curves form a significant class of explicit examples. In \cite{Thiago}, the correspondence between Bishop frames of Mannheim pairs was analyzed, and in \cite{TS, Thiago1} Mannheim curves were investigated in Heisenberg groups and in the special linear group, respectively. U\c{c}um, Camcı and \.{I}larslan \cite{Ucum} provided a new approach to Mannheim curves, including examples within the classes of Salkowski and anti-Salkowski curves.

Despite this extensive body of work, all explicit parametrizations of Mannheim curves found in the literature belong to the class of slant helices (or their subclasses). A natural question arises: \textit{Are there Mannheim curves in $\mathbb{R}^3$ that are neither helices nor slant helices, yet admit explicit parametrizations by elementary functions?}

In this paper, we answer this question affirmatively by constructing a new family of Mannheim curves using a systematic approach based on the angular coordinates of the tangent indicatrix. By making a specific choice for the auxiliary function $\delta(t)$, we obtain closed-form expressions for the coordinates of the curve, its Frenet vectors, curvatures, and the partner curve. The resulting curve is neither a helix nor a slant helix.

\section{Preliminaries}\label{sec:prelim}

Let $\mathbb{R}^3$ denote the Euclidean 3-space with the standard inner product 
$$\left\langle \cdot , \cdot\right\rangle = dx^2 + dy^2 + dz^2 \ ,$$ where $(x,y,z) \in \mathbb{R}^3$ is a rectangular coordinate system.

\subsection{Elementary curves}
A function is called \emph{elementary} if it can be obtained by elementary operations (addition, multiplication, division, composition, inversion, etc.) starting from a basic set of functions (polynomial, trigonometric, exponential, logarithmic, hyperbolic, etc.), excluding, for instance, elliptic and hypergeometric functions. Risch \cite{Risch1969} developed a method to decide whether a function admits an elementary integral; see also Geddes, Czapor, and Labahn \cite{Geddes1992} for a detailed treatment. A simple strategy to circumvent a negative result is to seek a new variable with respect to which the function becomes elementarily integrable.
However, when we deal with a parametrized curve described by two or more coordinates --- that is, a vector function whose coordinates depend on a single parameter --- there is no guarantee that there exists a single change of variable that makes the integrals of all coordinates elementary simultaneously.

\begin{example}$\displaystyle \int \sqrt{1+\cos^2(s)}ds$ is an elliptic function family. 
Consider $f(t)$ such that $\sqrt{1+\cos^2f(t)}.f'(t)=1$, then 

$$\displaystyle \int \sqrt{1+\cos^2(s)}ds = \displaystyle\int 1\,dt = t + \text{const}.$$
\end{example}

\begin{example}
Consider the curve $\alpha:\mathbb{R}\longrightarrow \mathbb{R}^3$ such that  
\begin{align*}
\dfrac{d\alpha}{ds}(s)&=\dfrac{1}{W(s)}\Bigl(-a\sin\left(t(s)\right),\ b\cos\left(t(s)\right),\ c\Bigr),\\
W(s)&=\sqrt{a^2\sin^2(t(s))+b^2\cos^2(t(s))+c^2},
\end{align*}
with $\alpha(0)=(a,0,0)$, and where $t(s)$ is the inverse of  
$$s(t)=\displaystyle \int_{0}^{t}\left(\sqrt{a^2\sin^2(*)+b^2\cos^2(*)+c^2}\right)\,d*.$$  
The coordinates of $\displaystyle \int \dfrac{d\alpha}{ds}(s)ds$ are not elementary. And we have   
$$\dfrac{d(\alpha\circ s)}{dt}(t)=\left(-a\sin(t),\,b\cos(t),\,c\right),$$  
hence  
$$(\alpha\circ s)(t)=\left(a\cos(t),\,b\sin(t),\,c\,t\right).$$
\end{example}

\subsection{General curves}

The \textit{prime} ``\ $\prime{}$\ ''  denotes derivative with respect to generic parameter $t$. And the \textit{dot} ``\ $\dot{}$\ ''  denotes derivative with respect to arc-length $s$. 

\begin{definition}
Let \(\alpha: I \subset \mathbb{R} \to \mathbb{R}^3\) be a regular curve. If \(\|\dot{\alpha}(s)\| = 1\) for all \(s \in I\), then \(s\) is called the \textit{arc length parameter} of \(\alpha\) and \(\alpha\) is said to be \textit{parameterized by arc length}.
\end{definition}

\begin{definition}
Let $\beta: J \subset \mathbb{R} \longrightarrow \mathbb{R}^{3}$, $\beta(t)$, with $t$ an arbitrary parameter, be a regular curve with non-zero curvature. The \textit{Frenet frame} (or Frenet–Serret apparatus) of $\beta$ consists of the following vector fields and scalar functions:
    \begin{enumerate}
    \item \textbf{Tangent:} 
    \begin{equation}\label{T x geral}
    T_{\beta}(t) = \dfrac{\beta'(t)}{\|\beta'(t)\|},    
    \end{equation}
    \item \textbf{Curvature:} 
    \begin{equation}\label{kappa x geral}
    \kappa_{\beta}(t) = \dfrac{\|\beta'(t) \times \beta''(t)\|}{\|\beta'(t)\|^3},    
    \end{equation}
    If $\kappa_\beta(t)\neq 0$, 
    
     \item \textbf{Principal normal:}
     \begin{equation}\label{N x geral}
     N_{\beta}(t) = \dfrac{T'_{\beta}(t)}{\|T'_{\beta}(t)\|}.  \end{equation}
     
    \item \textbf{Binormal:} 
    \begin{equation}\label{binormal x geral}
     B_{\beta}(t) = T_{\beta}(t) \times N_{\beta}(t),   
    \end{equation}

    \item \textbf{Torsion:} 
    \begin{equation}\label{tau x geral}
    \tau_{\beta}(t) = \dfrac{\det\bigl(\beta'(t),\,\beta''(t),\,\beta'''(t)\bigr)}{\|\beta'(t) \times   \beta''(t)\|^2}.    
    \end{equation}
\end{enumerate}
\end{definition}

\begin{definition}
Let $\beta: J \longrightarrow \mathbb{R}^3$, $\beta(t)$, be a regular curve and $t: I \longrightarrow J$, $t(v)$, a diffeomorphism. The curve $ \alpha= \beta \circ t: I \longrightarrow \mathbb{R}^3$, $ \alpha(v)= \beta (t(v))$, is called a \textit{reparametrization} of $\beta$ and $t(v)$ a \textit{change of parameter} of $\beta$. When $v=s$ is the arc length parameter of $\alpha$, $\alpha$ is called an \textit{arc length reparametrization} of $\beta$ and $t(s)$ is called a \textit{change of parameter by arc length} of $\beta$.
\end{definition}

\begin{proposition}Let $\alpha:I\longrightarrow \mathbb{R}^3$, $\alpha(v)=\beta(t(v))$, be a reparametrization of $\beta:J\longrightarrow \mathbb{R}^3$, $\beta(t)=\alpha(v(t))$, with $v=v(t)$ inverse of $t(v)$ and $\dfrac{dv}{dt}>0$. Then the Frenet apparatus of $\beta$ and $\alpha$ are related by:
\begin{enumerate}
    \item \textbf{Tangent at t:}
    \begin{equation}\label{tangente at t}
    T_{\beta}(t)=T_{\alpha}(v(t))
    \end{equation}
    \item \textbf{Curvature at t:}
    \begin{equation}\label{kappa at t}
    \kappa_{\beta}(t)=\kappa_{\alpha}(v(t))
    \end{equation}
    If $\kappa_\beta(t)\neq 0$,
    \item \textbf{Principal normal at t:}
    \begin{equation}\label{normal at t}
    N_{\beta}(t)=N_{\alpha}(v(t))
    \end{equation}
    \item \textbf{Binormal at t:}
    \begin{equation}\label{binormal at t}
    B_{\beta}(t)=B_{\alpha}(v(t))
    \end{equation}
    \item \textbf{Torsion at t:} 
    \begin{equation}\label{torsion at t}
    \tau_{\beta}(t)=\tau_{\alpha}(v(t))
    \end{equation}
\end{enumerate}
\end{proposition}

\begin{definition}
Let $\gamma:I\to S^{2}$, $\gamma(t)$, be a regular curve on the unit sphere. Its \emph{geodesic curvature} (with respect to the orientation $N(p)=p$) in $S^{2}$ is
\begin{equation}\label{geodesic curvature in S2}
\kappa^{\smash[b]{\gamma}}_{g}(t)=\frac{\det(\gamma(t),\gamma'(t),\gamma''(t))}{\|\gamma'(t)\|^{3}} .
\end{equation}
\end{definition}

\begin{remark}
The definition above is a particular case of the general definition of geodesic curvature: for a regular curve $\gamma$ on an oriented surface $S\subset\mathbb R^{3}$ with unit normal $N$,
\begin{equation}\label{geodesic curvature general}
\kappa^{\smash[b]{\gamma,S}}_{g}(t)=\frac{\langle N(\gamma(t))\times\gamma'(t),\,\gamma''(t)\rangle}{\|\gamma'(t)\|^{3}}.    
\end{equation}

Taking $S=S^{2}$ and $N(p)=p$ one obtains the definition above.
\end{remark}

\begin{proposition}
Let $\beta:I\to\mathbb R^{3}$ be a regular curve, $\beta(t)$,  with $\beta'\times\beta''\neq0$. The geodesic curvature, in $S^{2}$, of its tangent indicatrix $T_{\beta}(t)$ is

\begin{equation}\label{geodesic curvature and quotient}
\kappa^{\smash[b]{T_{\beta}}}_{g}(t)=\frac{\tau_\beta(t)}{\kappa_\beta(t)}.    
\end{equation}
\end{proposition}

\subsection{Helix and Slant Helix}
\begin{definition}[Lancret \cite{Lancret1802}]
A regular curve $\alpha: I \subset \mathbb{R} \to \mathbb{R}^{3}$ with non-zero curvature is a \emph{generalized helix} if its tangent $T_{\alpha}$ makes a constant angle with a fixed direction in space.
\end{definition}

\begin{definition} [Izumiya and Takeuchi \cite{IzumiyaTakeuchi2004}]
A regular curve $\alpha: I \subset \mathbb{R} \to \mathbb{R}^{3}$ with non-zero curvature is a \emph{slant helix} if its principal normal $N_{\alpha}$ makes a constant angle with a fixed direction in space.
\end{definition}

\begin{theorem}[Lancret \cite{Lancret1802} / Saint-Venant \cite{SaintVenant1845}]\label{helix characterization}
A curve $\alpha$ with curvature $\kappa_{\alpha} \neq 0$ is a generalized helix if and only if
\begin{equation}\label{generalhelix_equation}
\sigma_{0}(t) =\frac{\tau_{\alpha}}{\kappa_{\alpha}}= \text{constant}.
\end{equation}
\end{theorem}

\begin{theorem}[Izumiya and Takeuchi \cite{IzumiyaTakeuchi2004}]{(Slant helix characterization)}
\label{slant helix characterization}
A curve $\alpha$ is a slant helix if and only if
\begin{equation}\label{slanthelix_equation}
\sigma_{1}(s) =\frac{{\kappa_{\alpha}}^2}{({\kappa_{\alpha}}^2+{\tau_{\alpha}}^2)^{3/2}} \left( \frac{\tau_{\alpha}}{\kappa_{\alpha}} \right)^{\!\cdot} = \text{constant}.
\end{equation}
\end{theorem}

\subsection{Frenet equations in \texorpdfstring{$\mathbb{R}^3$}{R3}}\label{frenet equations R3}

\begin{theorem}[Frenet compatibility equations \cite{Oneill}]
Let $\alpha: I \to \mathbb{R}^3$ be a curve parametrized by arc-length $s$, with curvature $\kappa_\alpha>0$ and torsion $\tau_\alpha$. Then its Frenet frame $\{T_\alpha,N_\alpha,B_\alpha\}$ satisfies the Frenet equations
\begin{equation}\label{eq:frenet}
\begin{aligned}
\dot{T}_\alpha &= \kappa_\alpha N_\alpha,\\
\dot{N}_\alpha &= -\kappa_\alpha T_\alpha + \tau_\alpha B_\alpha,\\
\dot{B}_\alpha &= -\tau_\alpha N_\alpha,
\end{aligned}
\end{equation}
Conversely, given continuous functions $\kappa>0$ and $\tau$ on $I$, an initial point $p_0 \in \mathbb{R}^3$ and an initial orthonormal right-handed frame $\{T_0,N_0,B_0\}$ at a point $s_0 \in I$, there exists a unique curve $\alpha: I \to \mathbb{R}^3$ parametrized by arc-length, whose curvature and torsion are $\kappa$ and $\tau$, and whose Frenet frame at $s_0$ is $\{T_0,N_0,B_0\}$ with $\alpha(s_0)=p_0$.

Thus, the Frenet equations are the \textbf{compatibility equations} for the Frenet frame: an orthonormal moving frame $\{T,N,B\}$ with associated scalar functions $\kappa>0$ and $\tau$ is the Frenet frame of an arc-length parametrized curve if and only if it satisfies \eqref{eq:frenet} and its first vector is the derivative of the curve.
\end{theorem}

\subsection{Frenet apparatus in spherical coordinates with respect to a generic parameter}\label{sec:angular}

This subsection details the complete formulation of the Frenet frame for a space curve expressed in spherical coordinates, starting from its arc-length reparametrization and then considering a general parameter $t$.

Let $\alpha: I \subset \mathbb{R} \to \mathbb{R}^3$ be a smooth unit-speed curve with arc-length parameter $s$. Since its unit tangent vector field $T_\alpha(s) = \dot{\alpha}(s)$ inherently satisfies $\Vert{}T_\alpha(s)\Vert{} = 1$ for every $s \in I$, its trajectory is naturally constrained to the unit sphere $\mathbb{S}^2$. Consequently, we can seamlessly parametrize $T_\alpha(s)$ in terms of two spherical angular functions as follows:

\begin{equation}\label{Talpha}
T_{\alpha}(s)=(\cos\theta(s)\cos\phi(s),\ \sin\theta(s)\cos\phi(s),\ \sin\phi(s)),
\end{equation}
where $\theta(s)\in\mathbb{R}$ and $\phi(s) \in \left(-\dfrac{\pi}{2},\dfrac{\pi}{2}\right)$ are differentiable real-valued functions. 

\begin{proposition}\label{thm:arbitrary_param}
Let $\beta(t) = \alpha(s(t))$ be a smooth, orientation-preserving reparametrization of $\alpha$ satisfying $\dfrac{ds}{dt}>0$. In order to express the Frenet apparatus explicitly in terms of the general parameter $t$, we define, for convenience, the composite angular functions as
\begin{equation}\label{eq:ThetaPhi}
\Theta(t) = \theta(s(t)), \qquad \Phi(t) = \phi(s(t)).
\end{equation}
Throughout the remainder of the text a prime ``\ $'$\ '' denotes differentiation with respect to $t$. The Frenet apparatus of $\beta$ is given by:
\begin{equation}\label{eq:deriv_comp}
\beta'(t) = \frac{ds}{dt}\cdot \big( \cos\Theta\cos\Phi,\; \sin\Theta\cos\Phi,\; \sin\Phi \big),
\end{equation}
\begin{align}
T_\beta(t) &= \big( \cos\Theta\cos\Phi,\; \sin\Theta\cos\Phi,\; \sin\Phi \big), \label{eq:T_x}\\
N_\beta(t) &= \frac{\big( -\sin\Theta\cos\Phi\,\Theta' - \cos\Theta\sin\Phi\,\Phi',\; \cos\Theta\cos\Phi\,\Theta' - \sin\Theta\sin\Phi\,\Phi',\; \cos\Phi\,\Phi' \big)}{\sqrt{\Theta'^2\cos^2\Phi + \Phi'^2}}, \label{eq:N_x}\\
B_\beta(t) &= \frac{\big( \sin\Theta\,\Phi' - \cos\Theta\sin\Phi\cos\Phi\,\Theta',\; -\cos\Theta\,\Phi' - \sin\Theta\sin\Phi\cos\Phi\,\Theta',\; \cos^2\Phi\,\Theta' \big)}{\sqrt{\Theta'^2\cos^2\Phi + \Phi'^2}}, \label{eq:B_x}\\
\kappa_\beta(t) &= \frac{1}{\dfrac{ds}{dt}} \sqrt{\Theta'^2\cos^2\Phi + \Phi'^2}, \label{eq:kappa_x}\\
\tau_\beta(t) &= \frac{1}{\dfrac{ds}{dt}} \left\{ \Theta'\sin\Phi+ \left[ \arctan\left( \frac{\Phi^{\prime}}{\Theta^{\prime}\cos\Phi} \right) \right]^{{\prime}} \right\}, \label{eq:tau_x}
 \end{align}
where \eqref{eq:tau_x} is valid where $\Theta'\neq0$.
\end{proposition}
\begin{proof} Let $\alpha(s)$ be a curve parametrized by arc-length. Denoting its tangent vector by \eqref{Talpha}, we obtain
$$
\dot{T}_{\alpha} 
= \bigl( -\dot{\theta}\sin\theta\cos\phi - \dot{\phi}\cos\theta\sin\phi,\; \dot{\theta}\cos\theta\cos\phi - \dot{\phi}\sin\theta\sin\phi,\; \dot{\phi}\cos\phi \bigr).
$$

Thus, by definition and simple calculations, we obtain the curvature, the normal vector and the binormal vector given by the equations:

   \begin{equation}\label{curvaturaalpha}
   \kappa_{\alpha}(s) = \sqrt{(\dot{\theta})^2\cos^2\phi + (\dot{\phi})^2},    
   \end{equation}

   \begin{equation}\label{normalalpha}
   N_{\alpha}(s) = \dfrac{1}{\sqrt{(\dot{\theta})^2\cos^2\phi + (\dot{\phi})^2}} \Big( -\dot{\theta}\sin\theta\cos\phi - \dot{\phi}\cos\theta\sin\phi,\; \dot{\theta}\cos\theta\cos\phi - \dot{\phi}\sin\theta\sin\phi,\; \dot{\phi}\cos\phi \Big),
   \end{equation} 

   \begin{equation}\label{binormalalpha}
   B_{\alpha}(s) = \dfrac{1}{\sqrt{(\dot{\theta})^2\cos^2\phi + (\dot{\phi})^2}} \Big( \dot{\phi}\sin\theta - \dot{\theta}\cos\theta\sin\phi\cos\phi,\; -\dot{\phi}\cos\theta - \dot{\theta}\sin\theta\sin\phi\cos\phi,\; \dot{\theta}\cos^2\phi \Big).    
   \end{equation}

Moreover, since $\tau_{\alpha}$ = $-\langle \dot{B}_{\alpha}\,, N_{\alpha}\rangle$ (or using the third coordinate of third Frenet Equation \eqref{eq:frenet}), we obtain

  \begin{equation}\label{taualpha} 
   \tau_{\alpha} = \dot{\theta}\sin\phi +\dot{\psi}, \ \ \mbox{with} \ \ \psi  =  \arctan\left( \frac{\dot{\phi}}{\dot{\theta}\cos\phi} \right).
   \end{equation}

Since $\beta(t) = \alpha(s(t))$, by chain rule we have

$$\beta'(t)= \dot{\alpha}(s(t))\,\,\dfrac{ds}{dt}= T_{\alpha}(s(t))\,\,\dfrac{ds}{dt}.$$

Thus and from equation \eqref{Talpha}, we get  \eqref{eq:deriv_comp} .

From equation \eqref{eq:ThetaPhi}, we obtain 
$$\Theta^{\prime}(t)=\dot{\theta}(s(t)).\dfrac{ds}{dt}\ ,\  \Phi^{\prime}(t)=\dot{\phi}(s(t)).\dfrac{ds}{dt}.$$    

Thus and from equation \eqref{curvaturaalpha}
$$\kappa_\beta(t) =\dfrac{1}{\dfrac{ds}{dt}} \sqrt{\left[\dot{\theta}(s(t)).\dfrac{ds}{dt}\right]^2\cos^2\phi(s(t))+\left[\dot{\phi}(s(t)).\dfrac{ds}{dt}\right]^2},$$
hence we obtain \eqref{eq:kappa_x}, and consequently \eqref{eq:N_x} and \eqref{eq:B_x}.

Finally,
$$\dfrac{1}{\dfrac{ds}{dt}}\cdot\Theta'(t)\sin\Phi(t)=\dfrac{1}{\dfrac{ds}{dt}}\cdot \dfrac{ds}{dt}(t)\cdot \dot{\theta}(s(t))\cdot \sin\phi(s(t))= \dot{\theta}(s(t))\cdot \sin\phi(s(t)).$$

Let $f(s)=\arctan\left(\dfrac{\dot{\phi}(s)}{\dot{\theta}(s)\cos\phi(s)}\right)$, so that $\psi=f$ in \eqref{taualpha}. Since $\Theta'(t)=\dot{\theta}(s(t))\dfrac{ds}{dt}$ and $\Phi'(t)=\dot{\phi}(s(t))\dfrac{ds}{dt}$, the common factor $\dfrac{ds}{dt}$ cancels in the quotient and
\begin{gather*}
\arctan\left( \frac{\Phi'(t)}{\Theta'(t)\cos\Phi(t)} \right)=f(s(t)),\\
\text{hence}\qquad \dfrac{d}{dt}\left[ \arctan\left( \frac{\Phi'(t)}{\Theta'(t)\cos\Phi(t)} \right) \right]=\dot{f}(s(t))\,\dfrac{ds}{dt}(t).
\end{gather*}
Therefore
\begin{align*}
&\dfrac{1}{\dfrac{ds}{dt}}\left\{\Theta'(t)\sin\Phi(t) +\dfrac{d}{dt}\left[ \arctan\left( \frac{\Phi'(t)}{\Theta'(t)\cos\Phi(t)} \right) \right]\right\}\\
&\qquad=\dot{\theta}(s(t))\sin\phi(s(t)) +\dot{f}(s(t))=\tau_{\alpha}(s(t))=\tau_{\beta}(t).
\end{align*}
\end{proof}

By a coordinate system we mean a proper rotation of the axes (determinant $+1$): $\kappa_\beta$, $\tau_\beta$ and the Frenet frame (up to the rotation) do not change, and $(\Theta,\Phi)$ are the angular coordinates of $T_\beta$ in that system, which requires $T_\beta\neq\pm e_3$ along $\beta$. We write $D=\dfrac{ds}{dt}>0$ and assume that $\Theta$ and $\Phi$ are of class $C^2$ on an open interval $J$; in particular, a derivative that does not vanish has constant sign.

\begin{remark}[Straight lines and planar curves]\label{rem:planar}
Recall that, if $\kappa_\beta>0$, $\beta$ is planar if and only if $\tau_\beta\equiv0$, i.e.\ if and only if $B_\beta$ is constant.
\begin{enumerate}
\item[(0)] \textbf{Straight lines.} $\kappa_\beta\equiv0$ if and only if $\beta$ is a straight line, if and only if $\Theta$ and $\Phi$ are both constant. In this case $T_\beta$ is constant and $N_\beta$, $B_\beta$ and $\tau_\beta$ are not defined. Items (a), (b) and Corollary~\ref{cor:helix_slant} assume $\kappa_\beta>0$.
\item[(a)] $\beta$ is planar if and only if there is a coordinate system in which $\Phi\equiv0$. In this case $\kappa_\beta=|\Theta'|/D$, $N_\beta=\operatorname{sgn}(\Theta')(-\sin\Theta,\cos\Theta,0)$, $B_\beta=\operatorname{sgn}(\Theta')(0,0,1)$ and $\tau_\beta\equiv0$.
\item[(b)] If $\Theta\equiv\Theta_0$ is constant in some coordinate system, then $\beta$ is planar, and $\kappa_\beta=|\Phi'|/D$, $N_\beta=\operatorname{sgn}(\Phi')(-\cos\Theta_0\sin\Phi,-\sin\Theta_0\sin\Phi,\cos\Phi)$, $B_\beta=\operatorname{sgn}(\Phi')(\sin\Theta_0,-\cos\Theta_0,0)$, $\tau_\beta\equiv0$. Conversely, if $\beta$ is planar and there is a unit vector $w$ parallel to its plane with $\pm w\notin T_\beta(J)$, then $\Theta$ is constant in every coordinate system with $e_3=w$.
\end{enumerate}
\end{remark}

\begin{proof}(0) By \eqref{eq:kappa_x}, $\kappa_\beta=0$ if and only if $\Theta'\cos\Phi=0$ and $\Phi'=0$; since $\cos\Phi>0$, this means $\Theta'=\Phi'=0$, i.e.\ $T_\beta$ is constant. Then $\beta(t)=\beta(t_0)+\Big(\int_{t_0}^{t}\frac{ds}{d\tau}\,d\tau\Big)T_\beta$ is a straight line. Conversely, a straight line has constant $T_\beta$, hence $\Theta'=\Phi'=0$ (in a coordinate system with $T_\beta\neq\pm e_3$).
(a) If $\Phi\equiv0$, then $(\beta_3)'=D\sin\Phi=0$, so $\beta$ lies in a plane $z=\text{const}$. Conversely, if $\beta$ lies in a plane with unit normal $v$, then $\langle T_\beta,v\rangle=0$; taking $e_3=v$ gives $\sin\Phi=0$. For $\Phi\equiv0$ the formulas follow from \eqref{eq:kappa_x}, \eqref{eq:N_x}, \eqref{eq:B_x}, and $\tau_\beta=0$ from \eqref{eq:tau_x} (here $\Theta'\neq0$ because $\kappa_\beta>0$).

(b) Let $v_0=(-\sin\Theta_0,\cos\Theta_0,0)$. Then $\langle\beta',v_0\rangle=D\cos\Phi(-\cos\Theta_0\sin\Theta_0+\sin\Theta_0\cos\Theta_0)=0$, so $\langle\beta,v_0\rangle$ is constant. The formulas follow from \eqref{eq:kappa_x}, \eqref{eq:N_x}, \eqref{eq:B_x} with $\Theta'=0$; $B_\beta$ is constant, hence $\tau_\beta=0$ (formula \eqref{eq:tau_x} does not apply, since $\Theta'=0$). Conversely, let $e_3=w$. The unit normal $v$ of the plane is orthogonal to $w$, so $v=(-\sin a,\cos a,0)$, and $0=\langle T_\beta,v\rangle=\cos\Phi\,\sin(\Theta-a)$. Since $\cos\Phi>0$, $\Theta$ is a continuous function with values in $a+\pi\mathbb{Z}$, hence constant.
\end{proof}

\begin{corollary}[Helices and slant helices]\label{cor:helix_slant}
Let $\beta$ be as in Proposition~\ref{thm:arbitrary_param}, with $\kappa_\beta>0$. Recall that $\beta$ is a \emph{helix} if $T_\beta$ makes a constant angle with a fixed direction, and a \emph{slant helix} if $N_\beta$ does.
\begin{enumerate}
\item $\beta$ is a non-planar helix if and only if there is a coordinate system in which $\Phi\equiv\Phi_0$ with $\Phi_0\in(-\frac\pi2,\frac\pi2)\setminus\{0\}$. In this case $\Theta'\neq0$ and
\[
\kappa_\beta=\frac{\cos\Phi_0\,|\Theta'|}{D},\qquad \tau_\beta=\frac{\sin\Phi_0\,\Theta'}{D},\qquad \frac{\tau_\beta}{\kappa_\beta}=\operatorname{sgn}(\Theta')\tan\Phi_0
\]
(equivalently, by Lancret's theorem \cite{Lancret1802}, $\tau_\beta/\kappa_\beta$ is a nonzero constant). The case $\Phi_0=0$ is the planar case of Remark~\ref{rem:planar}.
\item $\beta$ is a slant helix and not a helix if and only if there is a coordinate system in which the third coordinate of $N_\beta$ is a nonzero constant $1/m$, that is,
\[
\frac{\Phi'\cos\Phi}{\sqrt{\Theta'^2\cos^2\Phi+\Phi'^2}}=\frac1m,
\]
equivalently, $\Phi'\neq0$, $\operatorname{sgn}m=\operatorname{sgn}\Phi'$ and
\[
\Theta'^2=\big(m^2-\sec^2\Phi\big)\Phi'^2 .
\]
In this case $|m|>1$, $m^2\cos^2\Phi>1$, $\Theta'\neq0$ and, with $\varepsilon=\operatorname{sgn}(\Theta'\Phi')$,
\[
\kappa_\beta=\frac{|m|\cos\Phi\,|\Phi'|}{D},\qquad
\tau_\beta=\frac{\varepsilon\,m^2\sin\Phi\cos\Phi\,\Phi'}{D\sqrt{m^2\cos^2\Phi-1}} .
\]
\end{enumerate}
\end{corollary}
\begin{proof}
(1) ($\Rightarrow$) Let $\langle T_\beta,v\rangle=c$ with $|v|=1$. If $c=0$, then $\langle\beta',v\rangle=0$ and $\beta$ is planar; so $c\neq0$. Since $\kappa_\beta>0$, $T_\beta$ is not constant, so $|c|<1$. Rotating $v$ to $e_3$ (admissible, since $\langle T_\beta,e_3\rangle=c\neq\pm1$) gives $\sin\Phi=c$, i.e.\ $\Phi\equiv\Phi_0=\arcsin c\neq0$. ($\Leftarrow$) If $\Phi\equiv\Phi_0$ then $\langle T_\beta,e_3\rangle=\sin\Phi_0$ is constant, so $\beta$ is a helix. By \eqref{eq:kappa_x} with $\Phi'=0$, $\kappa_\beta=\cos\Phi_0|\Theta'|/D>0$, so $\Theta'\neq0$; by \eqref{eq:tau_x}, $\tau_\beta=\sin\Phi_0\Theta'/D\neq0$, so $\beta$ is non-planar.

(2) ($\Rightarrow$) Let $\langle N_\beta,v\rangle=d$ with $|v|=1$. If $d=0$, then $\langle T_\beta',v\rangle=0$, so $\langle T_\beta,v\rangle$ is constant and $\beta$ is a helix; hence $d\neq0$. Rotate $v$ to $e_3$; this is admissible, because $T_\beta=\pm e_3$ would give $\langle N_\beta,e_3\rangle=0$. By \eqref{eq:N_x}, $\langle N_\beta,e_3\rangle=\Phi'\cos\Phi/\sqrt{\Theta'^2\cos^2\Phi+\Phi'^2}=d$. Squaring and putting $m=1/d$ gives $\Theta'^2=(m^2-\sec^2\Phi)\Phi'^2$, and $d\neq0$ gives $\Phi'\neq0$ with $\operatorname{sgn}\Phi'=\operatorname{sgn}d$. Conversely, if $\Phi'\neq0$, $\operatorname{sgn}m=\operatorname{sgn}\Phi'$ and $\Theta'^2=(m^2-\sec^2\Phi)\Phi'^2$, then $\Theta'^2\cos^2\Phi+\Phi'^2=m^2\cos^2\Phi\,\Phi'^2$, and so $\langle N_\beta,e_3\rangle=\Phi'\cos\Phi/(|m|\cos\Phi|\Phi'|)=1/m$.

\emph{The formulas.} From $\Theta'^2\ge0$, $m^2\ge\sec^2\Phi\ge1$; $m^2=1$ would force $\Phi\equiv0$, so $|m|>1$. Also $\Theta'^2\cos^2\Phi+\Phi'^2=m^2\cos^2\Phi\,\Phi'^2$ gives $\kappa_\beta$. Let $g=m^2\cos^2\Phi-1=\Theta'^2\cos^2\Phi/\Phi'^2\ge0$. If $g(t_0)=0$, then $\cos^2\Phi(t_0)=m^{-2}<1$, so $\Phi(t_0)\neq0$ and $g'(t_0)=-2m^2\cos\Phi\sin\Phi\,\Phi'\neq0$, which is impossible for $g\ge0$. Hence $g>0$ and $\Theta'\neq0$. Put $u=\sqrt g>0$. Then $\Theta'=\varepsilon\Phi'u/\cos\Phi$ with $\varepsilon=\operatorname{sgn}(\Theta'\Phi')$ constant, $\dfrac{\Phi'}{\Theta'\cos\Phi}=\dfrac{\varepsilon}{u}$ and $u^2+1=m^2\cos^2\Phi$. Since $u'=-m^2\cos\Phi\sin\Phi\,\Phi'/u$,
\[
\Big[\arctan\frac{\varepsilon}{u}\Big]'=-\frac{\varepsilon u'}{u^2+1}=\frac{\varepsilon\sin\Phi\,\Phi'}{u\cos\Phi},
\qquad
\Theta'\sin\Phi=\frac{\varepsilon\Phi'u\sin\Phi}{\cos\Phi},
\]
and \eqref{eq:tau_x} gives
\[
D\tau_\beta=\varepsilon\Phi'\tan\Phi\Big(u+\frac1u\Big)=\frac{\varepsilon\,m^2\sin\Phi\cos\Phi\,\Phi'}{u}.
\]

\emph{Not a helix.} By the above, $\tau_\beta/\kappa_\beta=\operatorname{sgn}(\Theta')|m|\sin\Phi/u$, and $(\sin\Phi/u)'=\cos\Phi\,\Phi'(m^2-1)/u^3\neq0$, because $u^2+m^2\sin^2\Phi=m^2-1$. So $\tau_\beta/\kappa_\beta$ is not constant. By part~(1) and Remark~\ref{rem:planar}, a helix has $\tau_\beta/\kappa_\beta$ constant, since this quotient does not depend on the coordinate system. Hence $\beta$ is not a helix.
\end{proof}

\subsection{Mannheim pair}

\begin{definition}\label{def:mannheim}
A curve $\alpha(s)$ is called a \emph{Mannheim curve} if there exist a differentiable curve $\alpha^*(s^*)$ and a reparametrization $s^*=s^*(s)$ such that, at corresponding points, the principal normal line of $\alpha$ coincides with the binormal line of $\alpha^*$. The curve $\alpha^*$ is called the \emph{Mannheim partner} of $\alpha$, and $(\alpha,\alpha^*)$ is called a \emph{Mannheim pair}. Equivalently,
\begin{equation}
    \overrightarrow{\alpha(s)\alpha^*(s^*)}// N_\alpha(s)// B_{\alpha^*}(s^*).
\end{equation}
\end{definition}

\begin{remark}\label{rem:mannheim-sign}
Throughout this paper, we assume without loss of generality that
\begin{equation}\label{fundamental hypotesis Mannheim Pair}
    N_\alpha(s)=-B_{\alpha^*}(s^*).
\end{equation}
\end{remark}

\begin{example}[\cite{Mendonca2026}]\label{Example Mannheim pair slanthelixXhelix} Mannheim pair (Slant helix, helix). 
\begin{align*}
\alpha(s)&= \dfrac{1}{2}
             \Bigl(
             \ln\left(\operatorname{sech}(s) \right)-\tanh^2 (s), \arcsin\left(\operatorname{sech}(s) \right)+\operatorname{sech}(s)\tanh(s),\\
&\qquad\qquad \sqrt{3}\ln\left(\operatorname{sech}(s) \right)+\frac{\sqrt{3}}{3}
             \Bigr),\ s<0,\\
\alpha^{*}(s^{*}) &=  \dfrac{1}{2}
             \Bigl(
             s^{*}-e^{2s^{*}}\ ,\ e^{s^{*}}\sqrt{1-e^{2s^{*}}}+\arcsin(e^{s^*})\ ,\ \sqrt{3}\ s^{*}
             \Bigr),\ s^{*}<0.
\end{align*}
 is a Mannheim pair with $\lambda=\dfrac{1}{\sqrt{3}}$.
 
 In fact, for $s^{*}(s)=\ln(\operatorname{sech}(s))$ with $s<0$, 

\begin{equation}
\alpha^{*}(s^{*}(s))-\alpha(s)=\frac{1}{\sqrt{3}}\cdot\dfrac{1}{2}\left(\sqrt{3}(2\tanh^2(s)-1), -2\sqrt{3}\operatorname{sech}(s)\tanh(s), -1 \right)    
\end{equation}
 
 \begin{equation}
 N_{\alpha}(s)=\frac{1}{2}\left(\sqrt{3}\left(2\tanh^2(s) - 1\right), -2\sqrt{3}\operatorname{sech}(s) \tanh(s), -1\right)  
 \end{equation}
 
\begin{equation}
 B_{\alpha^{*}}(s^{*}(s))=-\dfrac{1}{2}\left(\sqrt{3}\left(2\tanh^2(s)-1\right),  -2\sqrt{3}\operatorname{sech}(s)\tanh(s), -1 \right)  
\end{equation}

Then 
$$\alpha^{*}(s^{*}(s))-\alpha(s)=\dfrac{1}{\sqrt{3}}.N_{\alpha}(s)=-\dfrac{1}{\sqrt{3}}.B_{\alpha^{*}}(s^{*}(s)).$$
\end{example}

The next theorem collects the relations between the Frenet Apparatus of a Mannheim pair, obtained by straightforward calculations from Theorems~3 and~5 of \cite{OrbayKasap}.

\begin{theorem}[Orbay and Kasap \cite{OrbayKasap}]\label{lem:relations}Let $\{T_\alpha, N_\alpha, B_\alpha, \kappa_\alpha, \tau_\alpha\}$ and $\{T_{\alpha^*}, N_{\alpha^*}, B_{\alpha^*}, \kappa_{\alpha^*}, \tau_{\alpha^*}\}$ be the Frenet apparatus of the Mannheim pair $(\alpha, \alpha^*)$, respectively. Define $\psi(s, s^*) = \angle(T_\alpha(s), T_{\alpha^*}(s^*))$, and set $\bm{a}(s) = \psi(s, s^*(s))$, $\bm{A}(s^*) = \psi(s(s^*), s^*)$. There exists a nonzero constant $\lambda$ such that
\begin{equation}\label{eq:alpha_star}
\alpha^*(s^*) = \alpha(s(s^{*})) + \lambda N_\alpha(s(s^{*})), 
\end{equation}
or equivalently 
\begin{equation}\label{eq:alpha}
\alpha(s)=\alpha^*(s^{*}(s))+\lambda B_{\alpha^{*}}(s^{*}(s)).
   \end{equation}
And if $\kappa_{\alpha^*}\neq 0$ the following relations hold:
\begin{align}
T_{\alpha^*}(s^{*}) &= \cos\mathbf{A}(s^{*})\cdot T_\alpha(s(s^{*})) - \sin\mathbf{A}(s^{*})\cdot B_\alpha(s(s^{*})), \label{eq:T_star}\\
\nonumber 
\\
N_{\alpha^*}(s^{*}) &= \sin\mathbf{A}(s^{*})\cdot T_\alpha(s(s^{*})) + \cos\mathbf{A}(s^{*}) \cdot B_\alpha(s(s^{*})), \label{eq:N_star}\\
\nonumber 
\\
T_\alpha(s) &= \cos\bm{a}(s)\cdot T_{\alpha^*}(s^{*}(s)) + \sin\bm{a}(s)\cdot N_{\alpha^*}(s^{*}(s)), \label{eq:T_alpha}\\
\nonumber 
\\
B_\alpha(s) &= -\sin\bm{a}(s)\cdot T_{\alpha^*}(s^{*}(s)) + \cos\bm{a}(s)\cdot N_{\alpha^{*}}(s^{*}(s)), \label{eq:B_alpha}\\
\nonumber 
\\
\kappa_\alpha &= \frac{1}{\lambda}\sin^2\bm{a}, \qquad \tau_\alpha = -\frac{1}{\lambda}\sin\bm{a}\cos\bm{a}, \label{eq:kappa_tau_alpha}\\
\kappa_{\alpha^*} &= -\frac{d\bm{A}}{ds^*}, \qquad \tau_{\alpha^*} = -\frac{1}{\lambda}\tan\bm{A}, \label{eq:kappa_tau_star}\\
\cos\bm{a} &= \frac{ds^*}{ds}, \qquad \sec\bm{A} =\dfrac{ds}{ds^{*}}.  \label{eq:cos_a}
\end{align}
\end{theorem}

The following corollary follows immediately from the preceding theorem.

\begin{corollary} Let $\left(\alpha(s),\alpha^{*}(s^{*})\right)$ be a Mannheim pair with $\kappa_{\alpha^*}\neq0$ (suppose $\dfrac{ds^{*}}{ds}>0$). Then 

\begin{enumerate}
    \item \begin{equation}\label{tau* e tau e kappa}
\tau_{\alpha^{*}}(s^{*})=\dfrac{1}{\lambda}\dfrac{\kappa_{\alpha}(s(s^{*}))}{\tau_{\alpha}(s(s^{*}))}   
\end{equation}
\item   \begin{equation}\label{cos a s sin a s}
  \cos\bm{A}(s^{*})=\dfrac{1}{\sqrt{1+\left(\dfrac{\kappa_{\alpha}(s(s^{*}))}{\tau_{\alpha}(s(s^{*}))}\right)^2}},\qquad   \sin\bm{A}(s^{*})=-\dfrac{\kappa_{\alpha}(s(s^{*}))}{\tau_{\alpha}(s(s^{*}))}.\dfrac{1}{\sqrt{1+\left(\dfrac{\kappa_{\alpha}(s(s^{*}))}{\tau_{\alpha}(s(s^{*}))}\right)^2}}. 
  \end{equation}
\item 
\begin{equation}\label{ds ds*}
\dfrac{ds}{ds^{*}}(s^{*})=\sqrt{1+\left(\dfrac{\kappa_{\alpha}(s(s^{*}))}{\tau_{\alpha}(s(s^{*}))}\right)^2}    
\end{equation}
\item  \begin{equation}\label{kappa* e tau e kappa}
\kappa_{\alpha^{*}}(s^{*})=\dfrac{\dfrac{d}{ds^{*}}\left(\dfrac{\kappa_{\alpha}(s(s^{*}))}{\tau_{\alpha}(s(s^{*}))}\right)}{1+\left(\dfrac{\kappa_{\alpha}(s(s^{*}))}{\tau_{\alpha}(s(s^{*}))}\right)^2}  
\end{equation}
\item \textbf{Mannheim equation:}
\begin{equation}\label{mannheim equation}
\kappa_{\alpha}(s)^2+\tau_{\alpha}(s)^2=\dfrac{1}{\lambda}\kappa_{\alpha}(s)
\end{equation}
\item 
  \begin{equation}\label{cos A s* sin A s*}
  \cos\bm{a}(s) = \dfrac{1}{\sqrt{1+\lambda^2\tau_{\alpha^*}(s^{*}(s))^2}}, \qquad
\sin\bm{a}(s)= -\dfrac{\lambda\,\tau_{\alpha^*}(s^{*}(s))}{\sqrt{1+\lambda^2\tau_{\alpha^*}(s^{*}(s))^2}}
   \end{equation}
  \item 
\begin{equation}
\dfrac{ds^{*}}{ds}(s)=\dfrac{1}{\sqrt{1+\lambda^2\tau_{\alpha^{*}}(s^{*}(s))^2}}    
\end{equation}
\item 
\begin{equation}\label{kappa kappa* e tau*}
\kappa_{\alpha}(s)=\dfrac{\lambda \tau_{\alpha^{*}}(s^{*}(s))^2}{1+\lambda^2 \tau_{\alpha^{*}}(s^{*}(s))^2} 
\end{equation}
\item 
\begin{equation}\label{tau kappa* e tau*}
\tau_{\alpha}(s)=\dfrac{\tau_{\alpha^{*}}(s^{*}(s))}{1+\lambda^2 \tau_{\alpha^{*}}(s^{*}(s))^2}     
\end{equation}
\item  \textbf{Mannheim pair equation:} 
\begin{equation}\label{mannheim pair equation}
\kappa_{\alpha^{*}}(s^{*})=\dfrac{\lambda \dfrac{d}{ds^{*}}\left(\tau_{\alpha^{*}}(s^{*})\right)}{1+\lambda^2 \tau_{\alpha^{*}}(s^{*})^2}     
\end{equation}
\end{enumerate}   
\end{corollary}

\begin{proof}
Since $\frac{ds^*}{ds}>0$, the map $s\mapsto s^*(s)$ is a diffeomorphism onto an interval, so an identity at $s$ is the same as an identity at $s^*=s^*(s)$, and $\bm a(s)=\psi(s,s^*(s))=\bm A(s^*(s))$ by definition. Moreover $\cos\bm A=\frac{ds^*}{ds}>0$ (see \eqref{eq:cos_a}), so $\bm A\in(-\frac\pi2,\frac\pi2)$ and $\bm A=\arctan(-\lambda\tau_{\alpha^*})$ is differentiable. Since $\kappa_\alpha=\frac1\lambda\sin^2\bm a>0$ we have $\sin\bm a\neq0$, and $\cos\bm a=\frac{ds^*}{ds}>0$, so $\tau_\alpha\neq0$. Write $r=\kappa_\alpha/\tau_\alpha$. From \eqref{eq:kappa_tau_alpha},
\begin{equation}\label{eq:r_tan}
r(s)=\dfrac{\frac1\lambda\sin^2\bm a(s)}{-\frac1\lambda\sin\bm a(s)\cos\bm a(s)}=-\tan\bm a(s).
\end{equation}
\begin{enumerate}
\item By \eqref{eq:r_tan} and \eqref{eq:kappa_tau_star}, $\dfrac1\lambda r(s)=-\dfrac1\lambda\tan\bm a(s)=-\dfrac1\lambda\tan\bm A(s^*(s))=\tau_{\alpha^*}(s^*(s))$.
\item By \eqref{eq:r_tan}, $\sec^2\bm a=1+\tan^2\bm a=1+r^2$. Since $\cos\bm a>0$, $\cos\bm a=\dfrac{1}{\sqrt{1+r^2}}$ and $\sin\bm a=\tan\bm a\cos\bm a=-\dfrac{r}{\sqrt{1+r^2}}$. Evaluating at $s=s(s^*)$ gives $\cos\bm A(s^*)$ and $\sin\bm A(s^*)$.
\item By \eqref{eq:cos_a} and item~2, $\dfrac{ds}{ds^*}=\sec\bm A=\dfrac{1}{\cos\bm A}=\sqrt{1+r^2}$.
\item By \eqref{eq:r_tan}, $r(s(s^*))=-\tan\bm A(s^*)$. Differentiating with respect to $s^*$ and using \eqref{eq:kappa_tau_star},
$$\dfrac{d}{ds^*}\,r(s(s^*))=-\sec^2\bm A\,\dfrac{d\bm A}{ds^*}=\sec^2\bm A\,\kappa_{\alpha^*}(s^*),$$
and $\sec^2\bm A=1+r^2$ gives the formula.
\item It follows directly from \eqref{eq:kappa_tau_alpha}: $\kappa_\alpha^2+\tau_\alpha^2=\frac1{\lambda^2}\sin^2\bm a\,(\sin^2\bm a+\cos^2\bm a)=\frac1\lambda\kappa_\alpha$.
\item By \eqref{eq:kappa_tau_star}, $\tan\bm A(s^*)=-\lambda\tau_{\alpha^*}(s^*)$, hence $\tan\bm a(s)=\tan\bm A(s^*(s))=-\lambda\tau_{\alpha^*}(s^*(s))$. Since $\cos\bm a>0$, $\cos\bm a=\dfrac1{\sqrt{1+\lambda^2\tau_{\alpha^*}^2}}$ and $\sin\bm a=\tan\bm a\cos\bm a=-\dfrac{\lambda\tau_{\alpha^*}}{\sqrt{1+\lambda^2\tau_{\alpha^*}^2}}$.
\item By \eqref{eq:cos_a} and item~6.
\item By \eqref{eq:kappa_tau_alpha} and item~6, $\kappa_\alpha(s)=\dfrac1\lambda\sin^2\bm a=\dfrac1\lambda\dfrac{\lambda^2\tau_{\alpha^*}^2}{1+\lambda^2\tau_{\alpha^*}^2}=\dfrac{\lambda\tau_{\alpha^*}(s^*(s))^2}{1+\lambda^2\tau_{\alpha^*}(s^*(s))^2}$.
\item By \eqref{eq:kappa_tau_alpha} and item~6, $\sin\bm a=-\lambda\tau_{\alpha^*}\cos\bm a$, hence
$$\tau_\alpha(s)=-\dfrac1\lambda\sin\bm a\cos\bm a=\tau_{\alpha^*}\cos^2\bm a=\dfrac{\tau_{\alpha^*}(s^*(s))}{1+\lambda^2\tau_{\alpha^*}(s^*(s))^2}.$$
\item By \eqref{eq:kappa_tau_star}, $\dfrac{d\tau_{\alpha^*}}{ds^*}=-\dfrac1\lambda\sec^2\bm A\,\dfrac{d\bm A}{ds^*}=\dfrac1\lambda\sec^2\bm A\,\kappa_{\alpha^*}$, and $\sec^2\bm A=1+\tan^2\bm A=1+\lambda^2\tau_{\alpha^*}^2$ gives the formula.
\end{enumerate}
\end{proof}

\section{Mannheim pair equation with respect to a generic parameter}\label{sec:mannheim_x}
  
We can state below the characterization of a Mannheim pair with respect to any other general parameter $t, x$ of the curves.
Before that, we establish a result about the reparametrization of a function.

\begin{proposition}\label{proposicao reparametrizacao nula} A function $f:I\longrightarrow J$, $f(s)$, is null if and only if the function $g=f\circ h:L\subset \mathbb{R}\longrightarrow J$ is null, where $h:L\subset \mathbb{R}\longrightarrow I$ is a diffeomorphism.
\end{proposition}
\begin{proof} 
($\Rightarrow$)  If $f(s)=0$, $\forall s\in I$, then  $g(t)=f(h(t))=0$, $\forall t\in L$.

($\Leftarrow$)  If $g(t)=0$, $\forall t\in L$, then  $f(s)=g(h^{-1}(s))=0$, $\forall s\in I$.
\end{proof}

Hence, let $\beta(t) = \alpha(s(t))$ be a reparametrization of $\alpha$ with $\dfrac{ds}{dt}>0$, and $\beta^{*}(t) = \alpha^{*}(s^{*}(s(t)))$. Or consider a reparametrization $\beta^{*}(x) = \alpha^{*}(s^{*}(x))$ of $\alpha^{*}$ with $\dfrac{ds^{*}}{dx}>0$, and $\beta(x) = \alpha(s(s^{*}(x)))$.

\begin{proposition}
\begin{enumerate}  
    \item  
    $\alpha(s)$ is a Mannheim curve with constant $\lambda\neq0$ if and only if 
    \begin{equation}\label{Mannheim x geral}
    \kappa_{\beta}(t)=\lambda \bigl[ \kappa_{\beta}(t)^2+\tau_{\beta}(t)^2\bigr],  
    \end{equation}
    for any reparametrization $s=s(t)$.
    \item 
    $\alpha^{*}(s^{*})$ is a Mannheim partner with constant $\lambda\neq0$ if and only if 

\begin{equation}\label{eq 2 mannheim pair em x}
 \kappa_{\beta^{*}}(x)\,\dfrac{ds^{*}}{dx}=\dfrac{\lambda\,\dfrac{d}{dx}\bigl(\tau_{\beta^{*}}(x)\bigr)}{1+\lambda^2 \bigl[\tau_{\beta^{*}}(x)\bigr]^2},  
\end{equation}
for any reparametrization $s^{*}=s^{*}(x)$.
\end{enumerate}
\end{proposition}
\begin{proof}
\begin{enumerate}
    \item By the characterization of Mannheim curves \cite{Liu, OrbayKasap}, $\alpha(s)$ is a Mannheim curve with constant $\lambda$ if and only if $f\equiv0$, where $f(s)=\kappa_{\alpha}(s) -\lambda \bigl( \kappa_{\alpha}(s)^2+\tau_{\alpha}(s)^2\bigr).$ By Proposition~\ref{proposicao reparametrizacao nula}, $f(s)=0$ if and only if $f(s(t))=0$ for any reparametrization $s=s(t)$, i.e., 

$$\kappa_{\alpha}(s(t))=\lambda \bigl[ \kappa_{\alpha}(s(t))^2+\tau_{\alpha}(s(t))^2\bigr],$$
and this is \eqref{Mannheim x geral}, since $\kappa_\beta(t)=\kappa_\alpha(s(t))$ and $\tau_\beta(t)=\tau_\alpha(s(t))$.
    
\item By the characterization of Mannheim partners of Liu and Wang \cite{Liu} (the necessity is item~10 of the Corollary above), $\alpha^{*}(s^{*})$ is a Mannheim partner with constant $\lambda$ if and only if $f\equiv0$, where $f(s^{*})=\kappa_{\alpha^{*}}(s^{*})-\dfrac{\lambda\,\dfrac{d \tau_{\alpha^{*}}}{ds^{*}}(s^{*})}{1+\lambda^2 \tau_{\alpha^{*}}(s^{*})^2}.$ By Proposition~\ref{proposicao reparametrizacao nula}, $f(s^{*})=0$ if and only if $f(s^{*}(x))=0$ for any reparametrization $s^{*}=s^{*}(x)$, i.e., 
    \begin{equation}
    \kappa_{\alpha^{*}}(s^{*}(x))=\dfrac{\lambda\,\dfrac{d \tau_{\alpha^{*}}}{ds^{*}}(s^{*}(x))}{1+\lambda^2 \tau_{\alpha^{*}}(s^{*}(x))^2}. \end{equation}
    Multiplying both sides by \(\dfrac{ds^{*}}{dx}\) and using the chain rule, $\alpha^*(s^{*})$ is a Mannheim partner iff for any reparametrization $s^{*}=s^{*}(x)$, 
\begin{equation}
 \kappa_{\alpha^{*}}(s^{*}(x))\,\dfrac{ds^{*}}{dx}=\dfrac{\lambda\,\dfrac{d}{dx}\bigl(\tau_{\alpha^{*}}(s^{*}(x))\bigr)}{1+\lambda^2 \bigl[\tau_{\alpha^{*}}(s^{*}(x))\bigr]^2}, 
\end{equation} 
which is \eqref{eq 2 mannheim pair em x}.
\end{enumerate}
\end{proof}     

From one curve of the Mannheim pair, we obtain the other curve of the pair and its respective Frenet apparatus, taking as reference the parameter and the Frenet apparatus of the another curve. This is formalized in the following proposition.

\begin{proposition}\label{proposition of pair generic parameter}
Let $(\alpha(s),\alpha^{*}(s^{*}))$ be a Mannheim pair with $\kappa_{\alpha^*}\neq0$ and $\dfrac{ds^{*}}{ds}>0$. Then
\begin{enumerate}
    \item If $t$ is a generic parameter for the curve $\alpha$, with $\dfrac{ds}{dt}>0$, then
    \begin{enumerate}
    \item 
\begin{equation}\label{alpha* to s*os(x)}
\beta^*(t)= \beta(t) + \lambda N_\beta(t).    
\end{equation}
    \item 
    \begin{equation}\label{B*=-N to s(x)}
     B_{\beta^{*}}(t)=-N_{\beta}(t).   
    \end{equation}
    \item 
\begin{equation}\label{T* at s*sx}
T_{\beta^{*}}(t) = \dfrac{1}{\sqrt{1+\left(\dfrac{\kappa_{\beta}(t)}{\tau_{\beta}(t)}\right)^2}}\cdot T_\beta(t) + \dfrac{\kappa_{\beta}(t)}{\tau_{\beta}(t)}.\dfrac{1}{\sqrt{1+\left(\dfrac{\kappa_{\beta}(t)}{\tau_{\beta}(t)}\right)^2}}\cdot B_\beta(t).
\end{equation}

\item 
\begin{equation}\label{N* at s*sx}
N_{\beta^{*}}(t) = -\dfrac{\kappa_{\beta}(t)}{\tau_{\beta}(t)}.\dfrac{1}{\sqrt{1+\left(\dfrac{\kappa_{\beta}(t)}{\tau_{\beta}(t)}\right)^2}}\cdot T_\beta(t)+\dfrac{1}{\sqrt{1+\left(\dfrac{\kappa_{\beta}(t)}{\tau_{\beta}(t)}\right)^2}}\cdot B_\beta(t).  
\end{equation}
\item \begin{equation}
\tau_{\beta^{*}}(t)=\dfrac{1}{\lambda} \dfrac{\kappa_\beta(t)}{\tau_\beta(t)}.    
\end{equation}
\item 
\begin{equation}\label{kappa* e tau e kappa com t}
\kappa_{\beta^{*}}(t)=
\dfrac{\dfrac{d}{dt}\left(\dfrac{\kappa_{\beta}(t)}{\tau_{\beta}(t)}\right)}{\dfrac{ds}{dt}.\sqrt{1+\left(\dfrac{\kappa_{\beta}(t)}{\tau_{\beta}(t)}\right)^2}}.
\end{equation}
\end{enumerate}

 \item If $x$ is a generic parameter for the curve $\alpha^{*}$, with $\dfrac{ds^{*}}{dx}>0$, then
 \begin{enumerate}
 \item \begin{equation}\label{eq:alpha_x 2}
\beta(x)=\beta^{*}(x)+\lambda B_{\beta^{*}}(x).
\end{equation}
     \item  \begin{equation}\label{N=-B* to s*(x)}
     N_{\beta}(x)=-B_{\beta^{*}}(x).   
    \end{equation}
     \item 
    \begin{equation}\label{eq:T_alpha st*}
        T_\beta(x)=\dfrac{1}{\sqrt{1+\lambda^2\tau_{\beta^*}(x)^2}}\cdot T_{\beta^*}(x) - \dfrac{\lambda\,\tau_{\beta^*}(x)}{\sqrt{1+\lambda^2\tau_{\beta^*}(x)^2}}\cdot N_{\beta^*}(x).
        \end{equation}         
     \item 
     \begin{equation}\label{eq:B_alpha st*}
         B_\beta(x)=\dfrac{\lambda\,\tau_{\beta^*}(x)}{\sqrt{1+\lambda^2\tau_{\beta^*}(x)^2}}\cdot T_{\beta^*}(x) + \dfrac{1}{\sqrt{1+\lambda^2\tau_{\beta^*}(x)^2}}\cdot N_{\beta^*}(x).
         \end{equation}
     \item \begin{equation}
         \tau_{\beta}(x)=\dfrac{\tau_{\beta^{*}}(x)}{1+\lambda^2 \tau_{\beta^{*}}(x)^2} 
     \end{equation}
     \item \begin{equation}\label{kappa kappa* e tau* x}
\kappa_{\beta}(x)=\dfrac{\lambda \tau_{\beta^{*}}(x)^2}{1+\lambda^2 \tau_{\beta^{*}}(x)^2}.
\end{equation}
 \end{enumerate}
 \end{enumerate}

\end{proposition}
\begin{proof}
\begin{enumerate}
    \item 
    \begin{enumerate}
    \item From equation \eqref{eq:alpha_star} at $s=s(t)$
$$ \alpha^*(s^*(s(t))) = \alpha(s(t)) + \lambda N_\alpha(s(t)),$$
and by equation \eqref{normal at t}
$$ \beta^{*}(t) = \beta(t) + \lambda N_\beta(t).$$
    \item From equation \eqref{fundamental hypotesis Mannheim Pair} at $s=s(t)$
    $$B_{\alpha^{*}}(s^{*}(s(t)))=-N_{\alpha}(s(t)),$$
    and by equation \eqref{binormal at t}
     $$B_{\beta^{*}}(t)=-N_{\beta}(t).$$
    \item Let $r=\dfrac{\kappa_{\alpha}(s(t))}{\tau_{\alpha}(s(t))}$. From equations \eqref{eq:T_star}, \eqref{cos a s sin a s} at $s=s(t)$, since $\cos\bm A=\dfrac{1}{\sqrt{1+r^2}}$ and $\sin\bm A=-\dfrac{r}{\sqrt{1+r^2}}$,
$$T_{\alpha^*}(s^{*}(s(t))) = \cos\bm A\cdot T_\alpha(s(t))-\sin\bm A\cdot B_\alpha(s(t)) = \dfrac{1}{\sqrt{1+r^2}}\cdot T_\alpha(s(t))+ \dfrac{r}{\sqrt{1+r^2}}\cdot B_\alpha(s(t)),$$
and by equations \eqref{tangente at t}, \eqref{kappa at t}, \eqref{normal at t},  \eqref{binormal at t}, \eqref{torsion at t} we obtain \eqref{T* at s*sx}.
    \item Similarly, from equations \eqref{eq:N_star}, \eqref{cos a s sin a s} at $s=s(t)$,
$$N_{\alpha^*}(s^{*}(s(t))) = \sin\bm A\cdot T_\alpha(s(t))+\cos\bm A\cdot B_\alpha(s(t)) = -\dfrac{r}{\sqrt{1+r^2}}\cdot T_\alpha(s(t))+ \dfrac{1}{\sqrt{1+r^2}}\cdot B_\alpha(s(t)),$$
which gives \eqref{N* at s*sx}.
    \item From equation \eqref{tau* e tau e kappa} at  $s=s(t)$. 
    \item With $r$ as above, $\dfrac{d}{ds^*}=\dfrac{ds}{ds^*}\dfrac{dt}{ds}\dfrac{d}{dt}$, and $\dfrac{ds}{ds^*}=\sqrt{1+r^2}$ by \eqref{ds ds*}. Then \eqref{kappa* e tau e kappa} at $s=s(t)$ gives
$$\kappa_{\beta^*}(t)=\dfrac{\sqrt{1+r^2}\,\dfrac{1}{ds/dt}\,r'(t)}{1+r^2}=\dfrac{r'(t)}{\dfrac{ds}{dt}\sqrt{1+r^2}}.$$
    \end{enumerate}
    \item 
\begin{enumerate}
        \item  From equation \eqref{eq:alpha} at $s^{*}=s^{*}(x)$ 
$$\alpha(s(s^{*}(x)))=\alpha^*(s^{*}(x))+\lambda B_{\alpha^{*}}(s^{*}(x)),$$
therefore,          
$$\beta(x)=\beta^*(x)+\lambda B_{\beta^{*}}(x).$$
\item From equation \eqref{fundamental hypotesis Mannheim Pair} at $s^{*}=s^{*}(x)$
    $$B_{\alpha^{*}}(s^{*}(x))=-N_{\alpha}(s(s^{*}(x))),$$
    therefore 
     $$B_{\beta^{*}}(x)=-N_{\beta}(x).$$
\item By \eqref{eq:T_alpha} at $s=s(s^*(x))$, with $\bm a(s(s^*))=\bm A(s^*)$, $\cos\bm A=\dfrac{1}{\sqrt{1+\lambda^2\tau_{\alpha^*}^2}}$ and $\sin\bm A=-\dfrac{\lambda\tau_{\alpha^*}}{\sqrt{1+\lambda^2\tau_{\alpha^*}^2}}$ (equations \eqref{cos A s* sin A s*}),
   $$T_\alpha(s(s^{*}(x)))=\cos\mathbf{A}(s^{*}(x))\cdot T_{\alpha^*}(s^{*}(x)) + \sin\mathbf{A}(s^{*}(x))\cdot N_{\alpha^*}(s^{*}(x)),$$ 
and by \eqref{tangente at t}, \eqref{normal at t}, \eqref{torsion at t} we obtain \eqref{eq:T_alpha st*}.
\item  Similarly from equations \eqref{eq:B_alpha}, \eqref{cos A s* sin A s*} at $s^{*}=s^{*}(x)$,
$$B_\alpha(s(s^{*}(x)))=-\sin\bm A\cdot T_{\alpha^*}(s^{*}(x))+\cos\bm A\cdot N_{\alpha^*}(s^{*}(x)),$$
which gives \eqref{eq:B_alpha st*}. 
\item From equation \eqref{tau kappa* e tau*} at $s^{*}=s^{*}(x)$.
\item From equation \eqref{kappa kappa* e tau*} at $s^{*}=s^{*}(x)$.
\end{enumerate}
\end{enumerate}    
\end{proof}

\section{Auxiliary angle function to Mannheim condition}\label{sec:angular_mannheim}

Our goal is to choose a convenient parameter $t$ and find functions $\Theta(t)$ and $\Phi(t)$ such that $\beta$ satisfies the Mannheim condition and the resulting integrals are elementary.    

Define an auxiliary angle $\delta(t)\in(0,\pi)$ by
\begin{equation}\label{eq:cot_delta}
\cot\delta(t) = \dfrac{\Theta'(t)\cos\Phi(t)}{\Phi'(t)}.
\end{equation}

Suppose $\beta$ is not a helix, $\dfrac{ds}{dt} >0$ and $\Phi'(t)>0$.

We have $$\sqrt{\Theta'^2\cos^2\Phi + \Phi'^2} = \Phi'(t)\csc\delta(t).$$

And
\begin{equation}
\left[ \arctan\big( \dfrac{\Phi'(t)}{\Theta'(t)\cos\Phi(t)} \big)\right]' = \delta'(t).
\end{equation}

Thus, by equations \eqref{eq:kappa_x}, \eqref{eq:tau_x}, the curvature and torsion become  
\[
\kappa_\beta(t) = \frac{\Phi'(t)\csc\delta(t)}{\dfrac{ds}{dt}},\]
\[\tau_\beta(t) = \frac{1}{\dfrac{ds}{dt}} \big[ \Phi'(t)\tan\Phi(t) \cot\delta(t) + \delta'(t) \big].\]

Imposing the Mannheim condition \eqref{mannheim equation} yields
\[
\frac{1}{\left(\dfrac{ds}{dt}\right)^2} \left[ \Phi'^2\csc^2\delta + \big( \Phi'\tan\Phi\cot\delta + \delta' \big)^2 \right] = \frac{1}{\lambda} \frac{\Phi'\csc\delta}{ds/dt}.\]

Solving for $\dfrac{ds}{dt}$, we obtain
\[\frac{ds}{dt} = \lambda \csc\delta(t)\Phi'(t) \left[ 1 + \left(\kappa^{\smash[b]{T_{\beta}}}_{g}(t)\right)^2 \right],\]

where
\[
 \kappa^{\smash[b]{T_{\beta}}}_{g}(t) = \tan\Phi(t) \cos\delta(t) +\dfrac{\sin\delta(t) \delta'(t)}{\Phi'(t)}.\]
Note that $\kappa^{\smash[b]{T_{\beta}}}_{g}=\tau_\beta/\kappa_\beta$, so by \eqref{geodesic curvature and quotient} it is the geodesic curvature of the tangent indicatrix.

Consequently, the curvature and torsion become
\[
\kappa_\beta(t)= \frac{1}{\lambda \left[1 + \left(\kappa^{\smash[b]{T_{\beta}}}_{g}(t)\right)^2 \right]},\ 
\tau_\beta(t) = \frac{\kappa^{\smash[b]{T_{\beta}}}_{g}(t) }{\lambda \left[1+ \left(\kappa^{\smash[b]{T_{\beta}}}_{g}(t)\right)^2\right]}. 
\]

The function $\Theta(t)$ is recovered from \eqref{eq:cot_delta}:
\[
\Theta(t) = \int \Phi'(t)\sec\Phi(t) \cot\delta(t) \, dt+m.\]

Thus, we can state the following theorem, which summarizes the general parametrization of a Mannheim curve obtained by this method.
 
\begin{theorem}\label{thm:general_mannheim}
Let $\lambda>0$ and $m$ be real constants, and let $\delta(t)\in(0,\pi)$ and $\Phi(t)\in(-\frac\pi2,\frac\pi2)$, with $\Phi'(t)>0$ (the case $\Phi'<0$ is analogous), be differentiable functions. Define  
\begin{equation}\label{F(x) 1}
F(t) = \tan\Phi(t) \cos\delta(t) +\dfrac{\sin\delta(t) \delta'(t)}{\Phi'(t)},  
\end{equation}
\begin{equation}\label{Theta(x) 1}
\Theta(t) = \int \Phi'(t)\sec\Phi(t) \cot\delta(t) \, dt+m,   
\end{equation}
\begin{equation}\label{dsdx 1}
\frac{ds}{dt} = \lambda \csc\delta(t)\Phi'(t) \left( 1 + F(t)^2 \right),
\end{equation}
and assume $F(t)\neq0$.

Then the curve $\beta$ parametrized by $t$ as
\begin{equation}\label{d alpha s}
\frac{d\beta}{dt}(t) = 
\dfrac{ds}{dt}\begin{bmatrix}
\cos \Theta(t) \cos\Phi(t)\\[4pt]
\sin \Theta(t) \cos\Phi(t)\\[4pt]
\sin\Phi(t)
\end{bmatrix},
\end{equation}
is a Mannheim curve, with partner $\beta+\lambda N_\beta$, and with curvature and torsion given by
\begin{equation}\label{k e tau}
\kappa_\beta(t) = \frac{1}{\lambda\big(1+F(t)^2\big)},\qquad
\tau_\beta(t) = \frac{F(t)}{\lambda\big(1+F(t)^2\big)},\quad \text{with}\  F(t)=\kappa^{\smash[b]{T_{\beta}}}_{g}(t). 
\end{equation}

Conversely, let $\beta$ be a Mannheim curve that is not a helix. Then, near each point, there are a coordinate system and functions $\delta(t)\in(0,\pi)$, $\Phi(t)$ with $\Phi'(t)\neq0$, such that \eqref{F(x) 1}, \eqref{Theta(x) 1},  \eqref{dsdx 1}, \eqref{d alpha s} and \eqref{k e tau} are satisfied.
\end{theorem}

\begin{remark}We can make one choice: $\Phi(t)=\phi(s(t))=t$ ($t$ as latitude parameter), and establish the corresponding forms of the equations of a Mannheim curve given by this theorem by function $\delta(t)$. Some suggestions to obtain some other forms of the equations (according to convenience) are to take one of the following possibilities: $\Theta(t)=\theta(s(t))=t$, $s(t)=t$, $s(t)={\kappa_{\beta}}^{-1}(t)$, $s(t)={\tau_{\beta}}^{-1}(t)$, $s(t)={\sigma_{0}}^{-1}(t)$, $s(t)={\sigma_{1}}^{-1}(t)$, where $\sigma_{0}$ and $\sigma_{1}$ are given by \eqref{generalhelix_equation} and \eqref{slanthelix_equation}.
\end{remark}

\subsection{Special Family and new elementary example}

\begin{example}We set $\Phi(t)=t$ and choose the function $\delta(t)$ such that
\begin{equation}\label{eq:choice_delta}
\cot\delta(t) = \textbf{a}\cdot \cos(t).
\end{equation}
This choice is motivated by the fact that it makes $\Theta(t)$ linear.

\textit{Basic functions}: 

From \eqref{eq:choice_delta}, we obtain:
\begin{align}
\csc\delta &= \sqrt{1 +{\textbf{a}}^2\cos^2(t)}, & \sin\delta &= \frac{1}{\sqrt{1 +{\textbf{a}}^2\cos^2(t)}}, &\cos\delta &= \frac{{\textbf{a}}\cdot\cos(t)}{\sqrt{1 + {\textbf{a}}^2\cos^2(t)}}, \\
&&&\delta'(t) = \frac{{\textbf{a}}\cdot\sin(t)}{1 + {\textbf{a}}^2\cos^2(t)}. & \label{eq:delta_prime}
\end{align}

\textit{The function $\Theta(t)$:}

Using \eqref{Theta(x) 1}, 
\begin{equation}
\Theta(t) = \int \sec(t) \cdot \textbf{a}\cos(t) \, dt = \int \textbf{a}\, dt = \textbf{a}t+\textbf{b}.
\end{equation}

\textit{The function} $\kappa^{\smash[b]{T_{\beta}}}_{g}(t)$ \textit{and} $\dfrac{ds}{dt}$:

Substituting into \eqref{F(x) 1}
\begin{align}
\kappa^{\smash[b]{T_{\beta}}}_{g}(t)&= \frac{\textbf{a}\sin(t) (2 + {\textbf{a}}^2\cos^2(t))}{\left(1 + {\textbf{a}}^2\cos^2(t)\right)^{3/2}}. \label{eq:F_explicit}
\end{align}

Then, from \eqref{dsdx 1}
\begin{align}
\frac{ds}{dt} &= \lambda  \sqrt{1 + {\textbf{a}}^2\cos^2(t)} \left[ 1 + \frac{{\textbf{a}}^2\sin^2(t) \left(2+{\textbf{a}}^2\cos^2(t)\right)^2}{\left(1+{\textbf{a}}^2\cos^2(t)\right)^3} \right] . \label{eq:ds_dphi_explicit}
\end{align}

\textit{Curvature and torsion of} $\beta$:

From \eqref{k e tau} 

\begin{equation}
\kappa_\beta(t) = \frac{1}{\lambda}\, . \, \frac{\left(1 + {\textbf{a}}^2\cos^2 (t)\right)^3}{\left(1 + {\textbf{a}}^2\cos^2 (t)\right)^3 + {\textbf{a}}^2\sin^2 (t)\,\left(2 + {\textbf{a}}^2\cos^2 (t)\right)^2},    
\end{equation}
\begin{equation}
\tau_\beta(t) = \frac{1}{\lambda}\, . \, \frac{{\textbf{a}}\sin (t)\,\left(2+{\textbf{a}}^2\cos^2(t)\right)\,\left(1+{\textbf{a}}^2\cos^2 (t)\right)^{3/2}}{\left(1 + {\textbf{a}}^2\cos^2 (t)\right)^3 + {\textbf{a}}^2\sin^2(t)\,\left(2 + {\textbf{a}}^2\cos^2 (t)\right)^2}.
\end{equation}

\textit{Coordinates of $\beta$:}

The derivative of the curve with respect to $t$ is given by \eqref{eq:deriv_comp}:
\begin{equation}\label{eq:dalpha_dphi}
\frac{d\beta}{dt}(t) = \frac{ds}{dt} \big( \cos({\textbf{a}}t+{\textbf{b}})\cos(t),\; \sin({\textbf{a}}t+{\textbf{b}})\cos(t),\; \sin(t) \big).
\end{equation}
\end{example}

\begin{remark}
The choice ${\textbf{a}}=2$ belongs to the family ${\textbf{a}}=2k$, $k\in\mathbb Z\setminus\{0\}$, for which the coordinates of $\beta$ are elementary functions. Let $R=\sqrt{1+{\textbf{a}}^2\cos^2(t)}$. Then
\[
\frac{ds}{dt}=\lambda\left[\frac{{\textbf{a}}^2-1}{R}+\frac{2{\textbf{a}}^2+1}{R^3}+\frac{{\textbf{a}}^2+1}{R^5}\right].
\]
For ${\textbf{a}}=2k$, $\cos(2kt)$ is a polynomial in $\sin t$ and $\sin(2kt)=\sin t\,U(\cos t)$ with $U$ a polynomial. Taking $u=\sin t$ in the terms with $\cos(2kt)$, and $u=\cos t$ in the terms with $\sin(2kt)$ and in $\sin t\,dt$, the primitives of the components of $\frac{d\beta}{dt}$ reduce to integrals of polynomials in $u$ times powers of the single square root $R$, which are elementary. For ${\textbf{a}}$ odd, the terms with $\cos({\textbf{a}}t)$ produce a second square root $\sqrt{1-u^2}$ together with $R$, and one obtains in general elliptic integrals in the first two coordinates (the third one is elementary for every ${\textbf{a}}$). Moreover, for ${\textbf{a}}=2k$ we have $\frac{d\beta}{dt}(t+\pi)=-\frac{d\beta}{dt}(t)$, so $\beta$ is a closed curve ($\beta(t+2\pi)=\beta(t)$). Since $\tau_\beta=0$ exactly at $t\in\pi\mathbb Z$, it is a Mannheim curve on $(0,\pi)$ and on $(\pi,2\pi)$.
\end{remark}

\subsection*{Case \texorpdfstring{${\mathbf{a}}=2$ and ${\mathbf{b}}=0$}{a=2 and b=0}:}\label{exemplo 1}

\textit{Basic functions}: 

From $\cot\delta = 2\cos t$, we obtain:
\begin{align}
\csc\delta &= \sqrt{1 + 4\cos^2 t}, &
\sin\delta &= \frac{1}{\sqrt{1 + 4\cos^2 t}}, &
\cos\delta &= \frac{2\cos t}{\sqrt{1 + 4\cos^2 t}}, \\
&&\delta'(t) &= \frac{2\sin t}{1 + 4\cos^2 t}. \label{eq:delta_prime_2}
\end{align}

\textit{The function $\Theta(t)$:}

Using $\Theta(t)=\int\sec t\cdot 2\cos t\,dt = \int 2\,dt$ (with ${\textbf{b}}=0$):
\begin{equation}
\Theta(t)=2t.
\end{equation}

\textit{The function $F(t)$ and $\dfrac{ds}{dt}$:}

Substituting into $F(t)$ yields
\begin{align}
F(t)&= \frac{4\sin(t) (1 + 2\cos^2(t))}{\left(1 + 4\cos^2(t)\right)^{3/2}}. \label{eq:F_explicit_4}
\end{align}

Then
\begin{align}
\frac{ds}{dt} &= \lambda\sqrt{1 + 4\cos^2(t)} \left[ 1 + \frac{16\sin^2(t) \left(1+2\cos^2(t)\right)^2}{\left(1+4\cos^2(t)\right)^3} \right] . \label{eq:ds_dphi_explicit_4}
\end{align}

\textit{Curvature and torsion of $\beta$:}
\begin{align}
\kappa_\beta(t) &= \frac{(1+4\cos^2(t))^3}{\lambda \left(48\cos^4(t) + 60\cos^2(t) + 17\right)}, \label{eq:kappa_explicit}\\
\tau_\beta(t) &= \frac{4\sin(t) (1+2\cos^2(t)) (1+4\cos^2(t))^{3/2}}{\lambda \left(48\cos^4(t) + 60\cos^2(t) + 17\right)}. \label{eq:tau_explicit}
\end{align}

\textit{The derivative with respect to $t$ is:}

\begin{equation}
\frac{d\beta }{dt}(t) = \frac{ds}{dt}\bigl(\cos(2t)\cos t,\; \sin(2t)\cos t,\; \sin t\bigr).    
\end{equation}

\textit{Coordinates of $\beta$:}

Integrating each component (with $\lambda$ as a factor) we obtain the coordinates in closed form. Thus, the Mannheim curve $\beta$ parametrized by $t$ is:
\begin{equation}\label{eq:alpha_explicit}
\beta(t) = \lambda \begin{bmatrix} \frac{\sin(t)\,(240\cos^4(t) - 72\cos^2(t) -43)}{20(1 + 4\cos^2(t))^{3/2}} + \frac{15}{8}\arcsin\!\left( \frac{2\sin(t)}{\sqrt{5}} \right) \\ \\ \frac{\cos(t)\,(-144\cos^4(t) + 104\cos^2(t) + 45)}{12(1 + 4\cos^2(t))^{3/2}} - \frac{15}{8}\arcsinh(2\cos(t))  \\ \\ -\frac{2}{3}\frac{\cos(t)(21 + 74\cos^2(t))}{(1 + 4\cos^2(t))^{3/2}} - \frac{3}{2}\arcsinh(2\cos(t))  \end{bmatrix}.
\end{equation}

 \textit{Frenet frame of} $\beta$:

Using \eqref{eq:T_x}--\eqref{eq:B_x} with $\Theta(t)=2t$, $\Phi(t)=t$, we obtain:
\begin{align}
T_\beta(t) &= \big( \cos(2t)\cos(t),\; \sin(2t)\cos(t),\; \sin(t) \big), \label{eq:T_explicit}\\
N_\beta(t) &= \frac{1}{\sqrt{1+4\cos^2(t)}} \Big( (1-6\cos^2(t))\sin(t),\; 2(3\cos^2(t)-2)\cos(t),\; \cos(t) \Big), \label{eq:N_explicit}\\
B_\beta(t) &= \frac{1}{\sqrt{1+4\cos^2(t)}} \Big( 4\sin^3(t)\cos(t),\; 1-6\cos^2(t)+4\cos^4(t),\; 2\cos^2(t)\Big).
\label{eq:B_explicit}
\end{align}

\textit {The Mannheim partner curve $\beta^{*}=\alpha^{*}\circ s^{*}\circ s$ :}

Combining equation \eqref{alpha* to s*os(x)} of Proposition~\ref{proposition of pair generic parameter} with \eqref{eq:alpha_explicit} and \eqref{eq:N_explicit} gives the explicit coordinates of the partner curve at the corresponding point:

\begin{equation}\label{eq:alpha_star_explicit}
\beta^{*}(t) =  \lambda \begin{bmatrix} 
-\frac{\sin (t)\,\bigl(240\cos^4 (t)+112\cos^2 (t)+23\bigr)}
{20(1+4\cos^2 (t))^{3/2}}+\frac{15}{8}\arcsin\!\left(\frac{2\sin(t)}{\sqrt{5}}\right)\\ \\ 
\frac{\cos(t)\left(144\cos^4(t)-16\cos^2(t)-3\right)}
{12(1+4\cos^2(t))^{3/2}}
-\frac{15}{8}\operatorname{arcsinh}(2\cos(t))
  \\  \\
-\frac{\cos(t)\bigl(39+136\cos^2(t)\bigr)}
{3(1+4\cos^2(t))^{3/2}}
-\frac{3}{2}\operatorname{arcsinh}(2\cos(t)) 
\end{bmatrix}.
\end{equation}

\textit{Frenet vectors of $\beta^{*}$:}

From Proposition~\ref{proposition of pair generic parameter}, we obtain:
\begin{equation}
T_{\beta^{*}}(t)=
\,
\begin{bmatrix}
\dfrac{12\sin(t)\,\cos^3(t)}{\sqrt{48\cos^4(t)+60\cos^2(t)+17}}
\\[0.4em]
\\[0.4em]
\dfrac{-(12\cos^4(t)-6\cos^2(t)-1)}{\sqrt{48\cos^4(t)+60\cos^2(t)+17}}
\\[0.4em]
\\[0.4em]
\dfrac{2(3\cos^2(t)+2)}{\sqrt{48\cos^4(t)+60\cos^2(t)+17}}
\end{bmatrix}.
\end{equation}
 \begin{equation}
N_{\beta^{*}}(t) =-
\begin{bmatrix}
\displaystyle
\frac{\cos(t)\,\bigl(48\cos^4(t)-6\cos^2(t)-17\bigr)}
{\sqrt{(1+4\cos^2(t))\bigl(48\cos^4(t)+60\cos^2(t)+17\bigr)}}
\\[2.5em]
\displaystyle
\frac{\sin(t)\,\bigl(48\cos^4(t)+18\cos^2(t)-4\bigr)}
{\sqrt{(1+4\cos^2(t))\bigl(48\cos^4(t)+60\cos^2(t)+17\bigr)}}
\\[2.5em]
\displaystyle
\frac{\sin(t)}
{\sqrt{(1+4\cos^2(t))\bigl(48\cos^4(t)+60\cos^2(t)+17\bigr)}}
\end{bmatrix}.
\end{equation}
\begin{equation}
\scriptstyle B_{\beta^{*}}(t)=-\dfrac{1}{\sqrt{1+4\cos^2(t)}}
\left(
(1-6\cos^2(t))\sin(t),\;
2(3\cos^2(t)-2)\cos(t),\;
\cos(t)
   \right).     
\end{equation}

\textit{Illustration:}
\begin{figure}[htbp]
\centering
\includegraphics[width=1.0\textwidth]{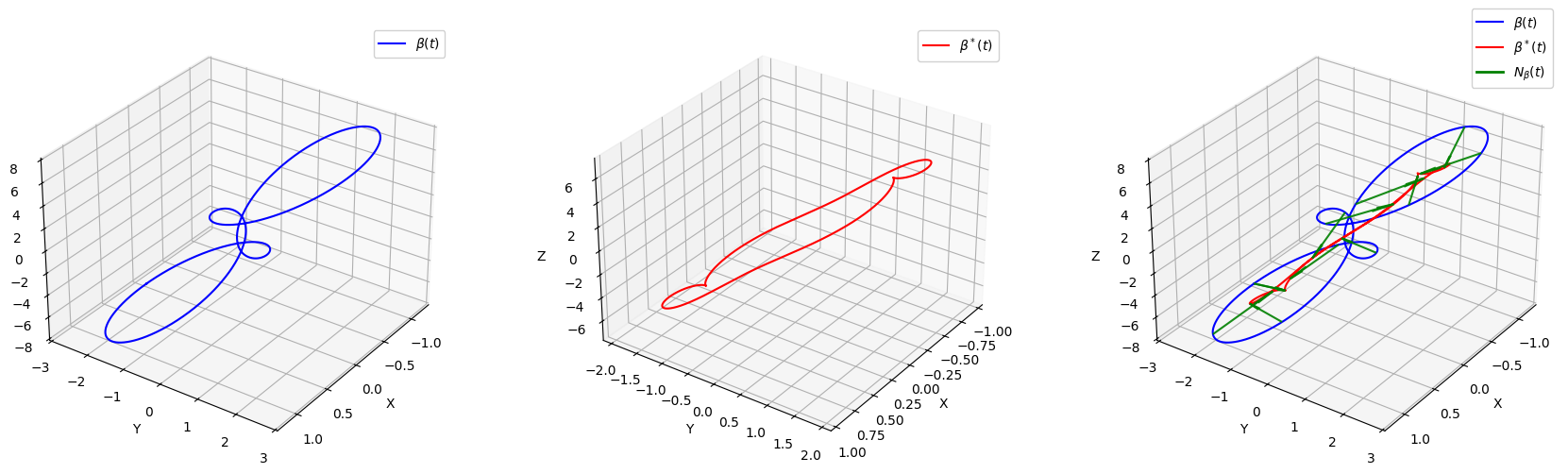}
\caption{Mannheim pair for ${\textbf{a}}=2$ and $\lambda=1$: original curve $\beta$ (blue) and its partner $\beta^*$ (red). Arrow (green) represents the principal normal $N_\beta$.}
\label{fig:mannheim_pair}
\end{figure}

\textit{Curvature, torsion and angle of the pair}:

From Proposition~\ref{proposition of pair generic parameter}, we obtain:
\begin{align}
\tau_{\beta^{*}}(t) & = \frac{(1+4\cos^2(t))^{3/2}}{4\lambda\sin(t)(1+2\cos^2(t))}, \label{eq:tau_star_explicit}\\
\kappa_{\beta^{*}}(t) &= \frac{3(1+4\cos^2(t))^3(3+2\cos^2(t))\,|\cot(t)|}{\lambda(1+2\cos^2(t))(48\cos^4(t)+60\cos^2(t)+17)^{3/2}}. \label{eq:kappa_star_explicit}
\end{align}

\textit{Non-helix and non-slant nature:} 

Using the explicit expressions for $\kappa_\beta$ and $\tau_\beta$ from \eqref{eq:kappa_explicit} and \eqref{eq:tau_explicit}, we compute
\begin{equation}
\kappa^{\smash[b]{T_{\beta}}}_{g}(t)=\frac{{\tau_\beta}(t)}{{\kappa_\beta}(t)} = \frac{4\sin(t)[1+2\cos^2(t)]}{[1+4\cos^2(t)]^{3/2}},
\end{equation}
which is clearly non-constant. From Theorem~\ref{helix characterization}, $\beta$ is not a helix.

There does not exist $\textbf{v}=(\bm{v}_{1},\bm{v}_{2},\bm{v}_{3})\in\mathbb{R}^3$ with $\|\textbf{v}\|=1$ and $d\in\mathbb{R}$ with $|\bm{d}|\leq 1$ such that $\langle N_{\beta}(t),\textbf{v}\rangle = \bm{d}$.

In fact, on the contrary,
\[
\langle N_{\beta}(t),\textbf{v}\rangle=\bm{d} \Longleftrightarrow
\bm{v}_{1}\cdot (1-6\cos^2(t))\sin(t) + \bm{v}_{2}\cdot2(3\cos^2(t)-2)\cos(t) + \bm{v}_{3}\cdot \cos(t) = \bm{d}\cdot\sqrt{1+4\cos^2(t)},
\]
this is not possible (no such constants exist). Then, for definition, $\beta$ is not a slant helix.

\section{Auxiliary angle function to Mannheim pair condition}\label{sec:angular_partner}

We can apply the same idea of the Mannheim condition to the curve $\alpha^{*}$ of a Mannheim pair with respect to a generic parameter $x$.
\begin{theorem}\label{thm:mannheim_partner_eucl}
Let \(\Phi^{*}(x)\) and \(\delta^{*}(x)\) be differentiable functions such that \({\Phi^{*}}'(x)\neq0\). Let \(\lambda\), \(\bm{p}\), and \(\bm{q}\) be constants. Define

\begin{equation}\label{Theta par}
\Theta^{*}(x)=\int {\Phi^{*}}'(x)\sec\Phi^{*}(x)\cot\delta^{*}(x)\,dx+\bm{q}, 
\end{equation}

\begin{equation}\label{ds* do par}
\frac{ds^{*}}{dx}(x)=\lambda\cot\!\left(\int{\Phi^{*}}'(x)\csc\delta^{*}(x)\,dx+\bm{p}\right)\Bigl({\Phi^{*}}'(x)\tan\Phi^{*}(x)\cot\delta^{*}(x)+{\delta^{*}}'(x)\Bigr).
\end{equation}
Then the curve \(\beta^{*}\) parametrized by $x$ as 
\begin{equation}\label{alpha* do par}
\frac{d\beta^{*}}{dx}(x)=\frac{ds^{*}}{dx}
\begin{bmatrix}
\cos\Theta^{*}(x)\cos\Phi^{*}(x)\\[4pt]
\sin\Theta^{*}(x)\cos\Phi^{*}(x)\\[4pt]
\sin\Phi^{*}(x)
\end{bmatrix},
\end{equation}
is a Mannheim pair curve with curvature and torsion given by
\begin{equation}\label{k e tau do par}
\kappa_{\beta^{*}}(x)=\frac{{\Phi^{*}}'(x)\csc\delta^{*}(x)}{\dfrac{ds^{*}}{dx}},\ 
\tau_{\beta^{*}}(x)=\frac{{\Phi^{*}}'(x)\tan\Phi^{*}(x)\cot\delta^{*}(x)+{\delta^{*}}'(x)}{\dfrac{ds^{*}}{dx}}.
\end{equation}
Conversely, every Mannheim pair curve is obtained in this way. That is, for every Mannheim pair curve \(\beta^{*}\), there exist \(\delta^{*}(x), \Phi^{*}(x)\) with \({\Phi^{*}}'(x)\neq0\), such that \eqref{Theta par}, \eqref{ds* do par},  \eqref{alpha* do par} and  \eqref{k e tau do par} are satisfied.
\end{theorem}

\begin{proof}
Suppose \(\beta^{*}\) is not a helix. We will assume \(\dfrac{ds^{*}}{dx} >0\) and \({\Phi^{*}}'(x)>0\).

Define an auxiliary angle \(\delta^{*}(x)\) by
\begin{equation}\label{eq:cot_delta_eucl}
\cot\delta^{*}(x) = \dfrac{{\Theta^{*}}'(x)\cos\Phi^{*}(x)}{{\Phi^{*}}'(x)}.
\end{equation}
Then
\begin{align*}
\sqrt{({\Theta^{*}}')^2\cos^2\Phi^{*} +({\Phi^{*}}')^2} &= \sqrt{\cot^2\delta^{*}\,{(\Phi^{*}}')^2\sec^2\Phi^{*}\cos^2\Phi^{*} + ({\Phi^{*}}')^2}\\
&= \sqrt{\cot^2\delta^{*}\,({\Phi^{*}}')^2+({\Phi^{*}}')^2} = {\Phi^{*}}'\csc\delta^{*},
\end{align*}
and
\begin{equation}
\left[ \arctan\!\big( \dfrac{{\Phi^{*}}'(x)}{{\Theta^{*}}'(x)\cos\Phi^{*}(x)} \big)\right]' = {\delta^{*}}'(x).
\end{equation}

Thus, by \eqref{eq:kappa_x} and \eqref{eq:tau_x}, we obtain
\[
\kappa_{\beta^{*}}(x) = \frac{{\Phi^{*}}'(x)\csc\delta^{*}(x)}{\dfrac{ds^{*}}{dx}(x)},
\]
\[
\tau_{\beta^{*}}(x) = \frac{1}{\dfrac{ds^{*}}{dx}} \bigl[ {\Phi^{*}}'(x)\tan\Phi^{*}(x)\cot\delta^{*}(x) + {\delta^{*}}'(x) \bigr].
\]

Imposing the equation of characterization of a Mannheim pair \eqref{eq 2 mannheim pair em x}, whose left-hand side is $\kappa_{\beta^{*}}\dfrac{ds^{*}}{dx}={\Phi^{*}}'\csc\delta^{*}$, and integrating it, we obtain $\arctan(\lambda\tau_{\beta^{*}})=\displaystyle\int {\Phi^{*}}'\csc\delta^{*}\,dx+\bm{p}$, that is,
\[
\frac{1}{\lambda}\tan\!\left( \int {\Phi^{*}}'(x)\csc\delta^{*}(x)\,dx + \bm{p} \right) = \frac{1}{\dfrac{ds^{*}}{dx}} \bigl[ {\Phi^{*}}'(x)\tan\Phi^{*}(x)\cot\delta^{*}(x) + {\delta^{*}}'(x) \bigr].
\]

Solving for \(\dfrac{ds^{*}}{dx}\), we obtain
\[
\dfrac{ds^{*}}{dx} = \lambda \cot\!\left( \int {\Phi^{*}}'(x)\csc\delta^{*}(x)\,dx + \bm{p} \right) \bigl[ {\Phi^{*}}'(x)\tan\Phi^{*}(x)\cot\delta^{*}(x) + {\delta^{*}}'(x) \bigr].
\]

The function \(\Theta^{*}(x)\) is recovered from \eqref{eq:cot_delta_eucl}:
\[
\Theta^{*}(x) = \int {\Phi^{*}}'(x)\sec\Phi^{*}(x)\cot\delta^{*}(x)\,dx + \bm{q}.\]
\end{proof}

Similarly to the Mannheim condition, we can make one choice: $\Phi^{*}(x)=\phi(s^{*}(x))=x$ ($x$ as latitude parameter), and establish the corresponding forms of the equations of a Mannheim pair curve given by this theorem by function $\delta^{*}(x)$. Some suggestions to obtain some other forms of the equations (according to convenience) are to take one of the following possibilities: $\Theta^{*}(x)=x$, $s^{*}(x)=x$, $s^{*}(x)={\kappa_{\alpha^{*}}}^{-1}(x)$, $s^{*}(x)={\tau_{\alpha^{*}}}^{-1}(x)$, $s^{*}(x)={\sigma_{0}}^{-1}(x)$, $s^{*}(x)={\sigma_{1}}^{-1}(x)$, where $\sigma_{0}$ and $\sigma_{1}$ are given by \eqref{generalhelix_equation} and \eqref{slanthelix_equation}.

\begin{example}In the case ${\textbf{a}}=2$, ${\textbf{b}}=0$ above, let $t\in(\frac\pi2,\pi)$, which is the interval where $N_\beta=-B_{\beta^*}$ and $\frac{ds^*}{dt}>0$ hold with $\kappa_{\beta^*}>0$. In this interval the latitude of $T_{\beta^*}$ is decreasing, so we apply Theorem~\ref{thm:mannheim_partner_eucl} in the coordinate system obtained by the rotation $(x,y,z)\mapsto(x,-y,-z)$, in which ${\Phi^{*}}'>0$. In this system, the corresponding Mannheim partner curve $\beta^{*}$ is determined by the functions
\begin{equation}
\Phi^{*}(x) = -\arcsin\!\left(\frac{2(3\cos^2(x)+2)}{\sqrt{48\cos^4(x)+60\cos^2(x)+17}}\right),    
\end{equation}
\begin{equation}
\delta^{*}(x) = \operatorname{arccot}\!\left(\cot(x) \; \sqrt{48\cos^4(x)+60\cos^2(x)+17}\right),    
\end{equation}
\begin{equation}
\Theta^{*}(x) = \pi+\arctan\!\left( \frac{12\cos^{4}(x)-6\cos^{2}(x)-1}{12\sin(x) \cos^{3}(x)} \right),
\end{equation}
with $x=t$ and $\operatorname{arccot}$ taking values in $(0,\pi)$.
\end{example}

\section*{Acknowledgements}
\noindent The authors would like to express their sincere thanks to the editor and the anonymous reviewers for their helpful comments and suggestions.

\section*{Declarations}
\noindent\textbf{Funding.} There is no funding for this work. \textbf{Availability of data.} Not applicable. \textbf{Competing interests.} The authors declare that they have no competing interests. \textbf{Author contributions.} All authors contributed equally to the writing of this paper. All authors read and approved the final manuscript.

\end{document}